\documentclass[12pt]{amsart}
\usepackage{amsmath}	
\usepackage{amssymb}
\usepackage{tikz}
\usepackage{accents}
\usetikzlibrary{calc}

\newtheorem{theorem}{Theorem}[section]
\newtheorem{lemma}[theorem]{Lemma}
\newtheorem{proposition}[theorem]{Proposition}
\newtheorem{corollary}[theorem]{Corollary}

\theoremstyle{remark}

\theoremstyle{definition}
\newtheorem{definition}[theorem]{Definition}

\numberwithin{equation}{section}

\newcommand{\ubar}[1]{\underaccent{\bar}{#1}}

\newcommand{\et}{\quad\mbox{and}\quad}

\newcommand{\bC}{\mathbb{C}}
\newcommand{\bN}{\mathbb{N}}

\newcommand{\bQ}{\mathbb{Q}}
\newcommand{\bR}{\mathbb{R}}
\newcommand{\bZ}{\mathbb{Z}}
\newcommand{\cC}{{\mathcal{C}}}

\newcommand{\cG}{{\mathcal{G}}}
\newcommand{\cK}{{\mathcal{K}}}

\newcommand{\cS}{{\mathcal{S}}}

\newcommand{\dist}{\mathrm{dist}}

\newcommand{\homega}{\hat{\omega}}

\newcommand{\proj}{\mathrm{proj}}
\newcommand{\psibot}{{\ubar{\psi}}}
\newcommand{\psitop}{{\bar{\psi}}}

\newcommand{\rbot}{{\underline{r}}}
\newcommand{\rtop}{{\overline{r}}}
\newcommand{\sbot}{{\underline{s}}}
\renewcommand{\stop}{{\overline{s}}}

\newcommand{\tbigwedge}{{\textstyle{\bigwedge}}}

\newcommand{\tP}{\tilde{P}}

\newcommand{\tuP}{\tilde{\uP}}

\newcommand{\ua}{\mathbf{a}}

\newcommand{\ue}{\mathbf{e}}

\newcommand{\uL}{\mathbf{L}}

\newcommand{\uP}{\mathbf{P}}

\newcommand{\uu}{\mathbf{u}}

\newcommand{\ux}{\mathbf{x}}

\newcommand{\uy}{\mathbf{y}}

\newcommand{\uz}{\mathbf{z}}

\newcommand{\disp}{\displaystyle}
\newcommand{\vol}{\mathrm{vol}}

\begin{document}

\baselineskip=14.6pt

\title[On the exponents of best approximation]
{Spectrum of the exponents\\ of best rational approximation}
\author{Damien ROY}
\address{
   D\'epartement de Math\'ematiques\\
   Universit\'e d'Ottawa\\
   585 King Edward\\
   Ottawa, Ontario K1N 6N5, Canada}
\email{droy@uottawa.ca}
\thanks{Work partially supported by NSERC}

\begin{abstract}
Using the new theory of W.~M.~Schmidt and L.~Summerer called
parametric geometry of numbers, we show that the going-up
and going-down transference inequalities of W.~M.~Schmidt
and M.~Laurent describe the full spectrum of the $n$ exponents
of best rational approximation to points in $\bR^{n+1}$.
\end{abstract}

\maketitle

%
%

\section{Introduction}
 \label{sec:intro}

The study of simultaneous rational approximation to points
in $\bR^n$ started with the basic estimates of Dirichlet
in \cite{Di1842}, and got full impetus in the years 1926-1938
through the transference theorems of Khintchine
\cite{Kh1926a,Kh1926b} and J\'arnik \cite{Ja1938}.
Around the same period of time, Mahler introduced new
tools in geometry of numbers and applied them to
these questions \cite{Ma1937}.  In 1967, Schmidt
\cite{Sc1967} enlarged the scope of the problem by studying
how a fixed vector subspace $A$ of $\bR^n$ or $\bC^n$ can
be approximated, in the ambient space, by vector subspaces
of a given dimension, defined over a fixed number field $K$.
For this purpose, he introduced the notion of height of such
a subspace and was lead to study several angles of
approximation depending on the dimension of $A$.  This important
work was recently revisited by Laurent \cite{La2009b}
in the case where $A$ is a one-dimensional subspace
of $\bR^n$ and where $K=\bQ$ is the field of rational numbers.
Then, $A$ is spanned by a single vector $\uu$ and
there is only one angle of approximation to consider.  This
lead Laurent to introduce a family of $n-1$ (ordinary)
exponents of approximation to points in $\bR^n$ (interpolating
between the two classical ones) and to recast the results of
Schmidt in that setting through a series of inequalities
relating these exponents.  The purpose of this paper is
to show that these inequalities describe the full spectrum
of these exponents, thus answering a question of Laurent
in \cite{La2009b}.  The proof uses the parametric geometry
of numbers introduced by Schmidt and Summerer in a series
of recent papers \cite{SS2009,SS2013a}, together with the
complements from \cite{R2015}.

In the next section, we recall the definition of these
exponents as well as the inequalities of Schmidt and
Laurent which link them, and we state our main result.
In Section \ref{sec:link}, we express these exponents
in the context of parametric geometry of numbers.
In Section \ref{sec:gen}, we introduce the notion of
generalized $n$-system and, using \cite{R2015},
we reduce the proof of our main result to a combinatorial
problem involving such systems.  We study a particular
family of generalized $n$-systems in Section \ref{sec:fam}
and use them in Section \ref{sec:proof} to complete the proof.

%
%

\section{Notation and main result}

For each integer $n\ge 1$, we view $\bR^n$ as an Euclidean
space for the standard inner product of two vectors $\ux$
and $\uy$ denoted $\ux\cdot\uy$.  We also view its exterior
algebra $\tbigwedge \bR^n = \oplus_{k=0}^n\tbigwedge^k\bR^n$
as an Euclidean space characterized by the property that,
for each orthonormal basis $\{\ue_1,\dots,\ue_n\}$ of $\bR^n$,
the products $\ue_{i_1}\wedge\cdots\wedge\ue_{i_k}$ with
$0\le k\le n$ and $1\le i_1<\cdots<i_k\le n$ form an
orthonormal basis of $\tbigwedge \bR^n$.  For each
$k=1,\dots,n$, we also denote by $\tbigwedge^k \bZ^n$
the lattice of $\tbigwedge^k \bR^n$ spanned by the products
$\ux_1\wedge\cdots\wedge\ux_k$ with $\ux_1,\dots,\ux_k\in\bZ^n$.
For $k=0$, we have $\tbigwedge^k \bR^n=\bR$ and we set
$\tbigwedge^0 \bZ^n=\bZ$.

We say that a vector subspace $S$ of $\bR^n$ is defined
over $\bQ$ if it is generated over $\bR$ by elements of $\bQ^n$.
For such a subspace $S$, a basis $\{\ux_1,\dots,\ux_k\}$
of $S\cap\bZ^n$ as a $\bZ$-module is also a basis of $S$
as a vector space over $\bR$ and, following Schmidt
\cite{Sc1967}, we define the \emph{height} of $S$ by
\[
 H(S) = \|\ux_1\wedge\cdots\wedge\ux_k\|,
\]
this being independent of the choice of $\{\ux_1,\dots,\ux_k\}$.

Given a non-zero point $\uu$ in $\bR^n$ and a non-zero
subspace $S$ of $\bR^n$, we define the \emph{projective
distance} between $\uu$ and $S$ by
\[
 \dist(\uu,S)
  = \frac{\|\uu\wedge\ux_1\wedge\cdots\wedge\ux_k\|}
         {\|\uu\|\,\|\ux_1\wedge\cdots\wedge\ux_k\|},
\]
where $\{\ux_1,\dots,\ux_k\}$ is any basis of $S$
over $\bR$.  Again this is independent of the choice
of the basis.  It is also given by
\[
 \dist(\uu,S) = \frac{\|\proj_{S^\perp}(\uu)\|}{\|\uu\|}
\]
where $\proj_{S^\perp}(\uu)$ denotes the orthogonal projection
of $\uu$ on the orthogonal complement $S^\perp$ of $S$.
Geometrically, it represents the sine of the smallest angle
between $\uu$ and a non-zero vector of $S$.  In the work of
Schmidt \cite{Sc1967}, this quantity is denoted
$\psi_1(A,S)$ or $\omega_1(A,S)$ where $A=\bR\uu$ is
the line spanned by $\uu$.

\begin{definition}
Let $n\ge 1$ and let $\uu\in\bR^{n+1}\setminus\{0\}$.
For each $j=0,\dots,n-1$, we denote by $\omega_j(\uu)$
(resp.\ $\homega_j(\uu)$) the supremum of all real numbers
$\omega$ such that, for arbitrarily large values of $Q$
(resp.\ for all sufficiently large values of $Q$), there
exists a vector subspace $S$ of $\bR^{n+1}$, defined
over $\bQ$, of dimension $j+1$, with
\[
 H(S)\le Q \et H(S)\dist(\uu,S) \le Q^{-\omega}.
\]
\end{definition}

In particular, we have $\omega_j(\uu)=\infty$ if $\uu$
belongs to a subspace $S$ of $\bR^{n+1}$ defined
over $\bQ$ of dimension $j+1$.  Otherwise, $\omega_j(\uu)$
is the supremum of all real numbers $\omega$ for which
there exist infinitely many subspaces $S$ of $\bR^{n+1}$,
defined over $\bQ$, of dimension $j+1$ with
\[
 \dist(\uu,S) \le H(S)^{-\omega-1}.
\]

Theorem 13 of \cite{Sc1967} shows that
\begin{equation}
 \label{results:eq:homega}
 \omega_j(\uu) \ge \homega_j(\uu) \ge \frac{j+1}{n-j}
 \quad
 (0\le j\le n-1).
\end{equation}
Thus $\omega_0(\uu),\dots,\omega_{n-1}(\uu)$ are
non-negative although possibly infinite.
The following result provides further inequalities
relating these exponents of approximation.

\begin{theorem}[Schmidt, Laurent]
 \label{results:thm:SL}
Let $n\in\bN^*$.  For any non-zero $\uu\in\bR^{n+1}$,
we have $\omega_0(\uu)\ge 1/n$ and
\begin{equation}
 \label{results:thm:SL:eq}
 \frac{j\omega_j(\uu)}{\omega_j(\uu)+j+1}
  \le \omega_{j-1}(\uu)
  \le \frac{(n-j)\omega_j(\uu)-1}{n-j+1}
 \quad
 (1\le j\le n-1),
\end{equation}
with the convention that the left-most ratio is equal
to $j$ if $\omega_j(\uu)=\infty$.
\end{theorem}

The left inequality in \eqref{results:thm:SL:eq} follows
from the Going-up-theorem of Schmidt \cite[Theorem 9]{Sc1967}
while the right inequality follows from
his Going-down-theorem \cite[Theorem 10]{Sc1967}.
This was observed by Laurent in \cite{La2009b} who also
introduced the exponents $\omega_j(\uu)$ for that purpose.
In the same paper, Laurent also notes that each individual
inequality in \eqref{results:thm:SL:eq} is best possible
because their combination yields Khinchine's transference
inequalities which are known to be best possible.
He also deduces that, for each fixed $j\in\{0,\dots,n-1\}$,
the spectrum of $\omega_j$, namely the set of all possible
values $\omega_j(\uu)$, is the full interval
$[(j+1)/(n-j),\infty]$.  Although he restricts
to points $\uu$ with $\bQ$-linearly independent coordinates,
this is also true over the the set of all $\uu\neq 0$
in view of \eqref{results:eq:homega}.

An independent proof for the right inequality
in \eqref{results:thm:SL:eq} is given by Laurent
in \cite{La2009b} and, for both inequalities,
by Bugeaud and Laurent in \cite{BL2010}, again for
points $\uu$ with $\bQ$-linearly independent coordinates.

Our main result below answers
a question of Laurent in \cite{La2009b} by showing
that the inequalities of Theorem \ref{results:thm:SL}
describe the set of all possible values for
the $n$-tuples $(\omega_0(\uu),\dots,\omega_{n-1}(\uu))$.

\begin{theorem}
 \label{results:thm:main}
Let $n\in\bN^*$.  For any $\omega_0,\dots,\omega_{n-1}
\in [0,\infty]$ satisfying $\omega_0\ge 1/n$ and
\begin{equation}
 \label{results:thm:main:eq}
 \frac{j\omega_j}{\omega_j+j+1}
  \le \omega_{j-1}
  \le \frac{(n-j)\omega_j-1}{n-j+1}
 \quad
 (1\le j\le n-1),
\end{equation}
there exists a point $\uu\in\bR^{n+1}$ with $\bQ$-linearly
independent coordinates such that
\[
 \omega_j(\uu)=\omega_j \et \homega_j(\uu)=\frac{j+1}{n-j}
 \quad
 (0\le j\le n-1).
\]
\end{theorem}

A description of the spectrum of the $2n$ exponents
$(\omega_0,\dots,\omega_{n-1},\homega_0,\dots,\homega_{n-1})$
was achieved by Laurent in \cite{La2009} for $n=2$ but,
for larger values of $n$, the question is open.
In \cite[\S1]{SS2013a}, Schmidt and Summerer propose
a different larger set of $2n+2$ exponents of approximation
to non-zero points in $\bR^n$.  Their spectrum also is
unknown, even for $n=2$.

%
%

\section{Link with parametric geometry of numbers}
\label{sec:link}

Fix $n\in\bN^*$ and $\uu\in\bR^{n+1}\setminus\{0\}$.
For each real number $Q\ge 1$, we form the convex body
\[
 \cC_\uu(Q)
 = \{\ux\in\bR^{n+1} \,;\, \|\ux\|\le 1, \ |\ux\cdot\uu| \le Q^{-1}\}
\]
and, for $j=1,\dots,n+1$, we denote by
$\lambda_j\big(\cC_\uu(Q)\big)$
its $j$-th minimum, namely the smallest real number
$\lambda>0$ such that $\lambda \cC_\uu(Q)$ contains
at least $j$ linearly independent points of $\bZ^{n+1}$.
On the model of Schmidt and Summerer in \cite{SS2013a},
we define
\[
 L_{\uu,j}(q) = \log \lambda_j(\cC_\uu(e^q))
 \quad
 (q\ge 0,\ 1\le j\le n+1),
\]
and we form the map $\uL_\uu\colon[0,\infty)\to \bR^{n+1}$
given by
\[
 \uL_\uu(q) = (L_{\uu,1}(q),\dots,L_{\uu,n+1}(q))
 \quad
 (q\ge 0).
\]
For each $j=1,\dots,n+1$, we also define
\[
 \psibot_{j}(\uu)
  = \liminf_{q\to\infty} \frac{L_{\uu,1}(q)+\cdots+L_{\uu,j}(q)}{q}
\et 
 \psitop_{j}(\uu)
 = \limsup_{q\to\infty} \frac{L_{\uu,1}(q)+\cdots+L_{\uu,j}(q)}{q}.
\]
The following result connects these quantities to those from
the previous section.

\begin{proposition}
 \label{link:prop:omega}
Let $n$ and $\uu$ be as above.  For each $j\in\{0,\dots,n-1\}$, we have
\[
 \omega_j(\uu)=\frac{1}{\psibot_{n-j}(\uu)}-1
 \et
 \homega_j(\uu)=\frac{1}{\psitop_{n-j}(\uu)}-1.
\]
\end{proposition}

The proof relies on the following alternative definition
for the exponents $\omega_j(\uu)$ and $\homega_j(\uu)$.

\begin{lemma}[Bugeaud-Laurent]
 \label{link:lemma:BL}
Let $j\in\{0,\dots,n-1\}$.  Then $\omega_j(\uu)$
(resp.\ $\homega_j(\uu)$) is the supremum of all real
numbers $\omega$ such that the inequalities
\begin{equation}
 \label{link:lemma:BL:eq}
 \|\uz\|\le Q \et \|\uz\wedge\uu\|\le Q^{-\omega}
\end{equation}
have a non-zero solution $\uz\in\tbigwedge^{j+1}\bZ^{n+1}$
for arbitrarily large values of $Q$
(resp.\ for all sufficiently large values of $Q$).
\end{lemma}

This is proved in \cite[\S4]{BL2010} for $\omega_j(\uu)$
in the case where $\uu$ has $\bQ$-linearly independent
coordinates but the same argument extends with little
change to any non-zero vector $\uu$ and applies also to
$\homega_j(\uu)$.

\begin{proof}[Proof of Proposition \ref{link:prop:omega}]
For each $j=0,\dots,n-1$ and each $Q\ge 1$, define
\[
 \cK^{(j+1)}_\uu(Q)
  = \{\uz\in\tbigwedge^{j+1}\bR^{n+1} \,;\,
      \|\uz\|\le Q, \ \|\uz\wedge\uu\|\le 1 \}.
\]
According to \cite[\S4, Lemma 3]{BL2010}, $\cK^{(j+1)}_\uu(Q)$
is comparable to the $(j+1)$-th compound body
of $\cK^{(1)}_\uu(Q)$.  As the latter is in turn
comparable to the dual of $\cC_\uu(Q)$ and as
$\vol(\cC_\uu(Q))\asymp Q^{-1}$, it follows
that the first minimum of $\cK^{(j+1)}_\uu(Q)$ with
respect to $\tbigwedge^{j+1}\bZ^{n+1}$ satisfies
\begin{align*}
 \lambda_1\big(\cK^{(j+1)}_\uu(Q)\big)
  &\asymp \lambda_1\big(\cK^{(1)}_\uu(Q)\big)\cdots
          \lambda_{j+1}\big(\cK^{(1)}_\uu(Q)\big) \\
  &\asymp \lambda_{n+1}\big(\cC_\uu(Q)\big)^{-1}\cdots
          \lambda_{n-j+1}\big(\cC_\uu(Q)\big)^{-1} \\
  &\asymp  Q^{-1} \lambda_1\big(\cC_\uu(Q)\big)
          \cdots\lambda_{n-j}\big(\cC_\uu(Q)\big)
   \qquad (0\le j\le n-1,\ Q\ge 1),
\end{align*}
with implied constants depending only on $n$ and $\uu$
(see \cite[\S\S14--15]{GL1987}).
Now the convex body of $\tbigwedge^{j+1}\bR^{n+1}$
defined by the conditions \eqref{link:lemma:BL:eq}
is $Q^{-\omega}\cK_\uu^{(j+1)}(Q^{\omega+1})$,
so its first minimum with respect to $\tbigwedge^{j+1}\bZ^{n+1}$
is bounded below and above by products of
\begin{equation}
 \label{link:prop:omega:eq}
 Q^{-1} \lambda_1\big(\cC_\uu(Q^{\omega+1})\big)\cdots
        \lambda_{n-j}\big(\cC_\uu(Q^{\omega+1})\big)
\end{equation}
with positive constants that are independent of $Q$
and of $\omega$, when $Q\ge 1$ and $\omega+1\ge 0$.
By Lemma \ref{link:lemma:BL}, $\omega_j(\uu)$
(resp.\ $\homega_j(\uu)$) is the supremum of all real
numbers $\omega$ for which this first minimum is $\le 1$
for arbitrarily large values of $Q$ (resp.\ for all
sufficiently large values of $Q$).
Since $\omega_j(\uu)\ge\homega_j(\uu)\ge 0$
by \eqref{results:eq:homega}, we may restrict to values
of $\omega$ with $\omega>-1$ and thus $\omega_j(\uu)$
(resp.\ $\homega_j(\uu)$) is also the supremum of all
real numbers $\omega>-1$ for which the product
in \eqref{link:prop:omega:eq} is $\le 1$ for arbitrarily
large values of $Q$ (resp.\ for all sufficiently large
values of $Q$).  However, upon writing $Q=e^{q/(\omega+1)}$
with $q>0$, the condition that this product is $\le 1$
amounts to
\[
  \frac{L_{\uu,1}(q)+\cdots+L_{\uu,n-j}(q)}{q}
  \le \frac{1}{\omega+1},
\]
and the conclusion follows.
\end{proof}

%
%

\section{Generalized $n$-systems}
\label{sec:gen}

Fix an integer $n\ge 2$.  In \cite[\S3]{SS2013a},
Schmidt and Summerer propose the notion of
$(n,0)$-system as a model for the general behavior of the maps
$\uL_\uu\colon[0,\infty)\to\bR^n$ attached to non-zero points
$\uu\in\bR^n$.  The following is a version of this
adapted to the present context.

\begin{definition}
 \label{gen:def1}
Let $I$ be a subinterval of $[0,\infty)$ with non-empty interior.
An $n$-system on $I$ is a continuous piecewise linear map
$\uP=(P_1,\dots,P_n)\colon I\to\bR^n$ with the following
properties.
\begin{itemize}
 \item[(S1)] For each $q\in I$, we have
   $0\le P_1(q)\le\cdots\le P_n(q)$ and $P_1(q)+\cdots+P_n(q)=q$.
 \item[(S2)] If $H$ is a non-empty open subinterval of $I$
   on which $\uP$ is differentiable, then there is an integer
   $r$ with $1\le r\le n$ such that $P_r$ has slope $1$ on
   $H$ while the other components $P_j$ of $\uP$ with $j\neq r$
   are constant on $H$.
 \item[(S3)] If $q$ is an interior point of $I$ at which $\uP$
   is not differentiable and if the integers $r$ and $s$ for which
   $P'_r(q^-)=P'_s(q^+)=1$ satisfy $r<s$, then we have
   $P_{r}(q)=P_{r+1}(q)=\cdots=P_{s}(q)$.
\end{itemize}
\end{definition}

Here, the condition that $\uP\colon I\to \bR^n$ is piecewise linear
means that the set $D$ of points of $I$ at which $\uP$ is not
differentiable (including the boundary points of $I$ that lie in $I$)
is a discrete subset of $I$, and that the derivative of $\uP$ is
locally constant on $I\setminus D$.  Such a map admits a left
derivative $\uP'(q^-)$ at each $q\in I$ with $q\neq \inf I$, and a right
derivative $\uP'(q^+)$ at each $q\in I$ with $q\neq \sup I$.  The slope
of a component $P_r$ of $\uP$ on an open subinterval $H$ of $I\setminus D$
means the constant value of its derivative on $H$ or equivalently
the slope of its graph over $H$.

Figure \ref{gen:fig:S3} shows the
combined graph of the functions $P_r,\dots,P_s$ over a neighborhood
of $q$ under the hypotheses of Condition (S3).  In that case,
the functions $P_{r+1},\dots,P_{s}$ coincide to the left of
$q$, while $P_{r},\dots,P_{s-1}$ coincide to its right.

\begin{figure}[h]
     \begin{tikzpicture}[xscale=0.6,yscale=0.3]
       \node[draw,circle,inner sep=1pt,fill] at (6,4) {};
         \draw[dashed] (6,4)--(6,1) node[below]{$q$};
       \draw[thick] (3,1) -- (9,7);
       \draw[thick] (3,4) -- (9,4);
         \node[left] at (3,1) {$P_{r}$};
         \node[left] at (3,4) {$P_{r+1}=\cdots=P_{s}$};
         \node[right] at (9,4) {$P_{r}=\cdots=P_{s-1}$};
         \node[right] at (9,7) {$P_{s}$};
     \end{tikzpicture}
\caption{Illustration for Condition (S3).}
\label{gen:fig:S3}
\end{figure}
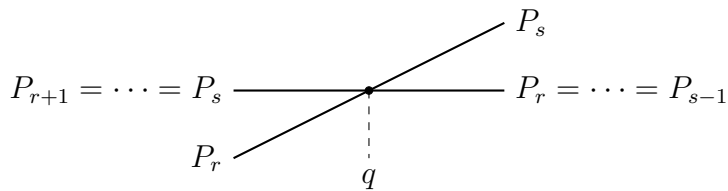

For an interval $I$ of the form $[q_0,\infty)$ with $q_0\ge 0$,
the above notion of an $n$-system on $I$ is the same as that of an
$(n,0)$-system on $I$ according to Definition 2.8 of \cite{R2015}.
These $n$-systems have the following approximation property.

\begin{theorem}
 \label{gen:thm:nsys}
For each non-zero point $\uu\in\bR^n$, there exists $q_0\ge 0$
and an $n$-system $\uP$ on $[q_0,\infty)$ such that
$\uL_\uu-\uP$ is bounded on $[q_0,\infty)$.  Conversely, for
each $n$-system $\uP$ on an interval $[q_0,\infty)$ with $q_0\ge 0$,
there exists a non-zero point $\uu\in\bR^n$ such that
$\uL_\uu-\uP$ is bounded on $[q_0,\infty)$.
\end{theorem}

The first assertion follows from \cite[Theorem 1.3]{R2015}
which states the same approximation property for a more
restricted class of maps called \emph{rigid $n$-systems},
while its converse follows from \cite[Theorem 8.1]{R2015}.
The fact that we deal with non-zero points $\uu$ of $\bR^n$
instead of unit vectors like in \cite{R2015} does not matter
for such a qualitative statement.

The goal of this section is to show that this approximation
property extends to the larger class of generalized
$n$-systems which we will define below.  In view of Proposition
\ref{link:prop:omega}, this reduces the determination of
the spectrum of any subsequence of
$(\omega_0,\dots,\omega_{n-1},\homega_0,\dots,\homega_{n-1})$
over $\bR^{n+1}\setminus\{0\}$ to a problem about generalized
$(n+1)$-systems.

The next two lemmas provide examples of $n$-systems.

\begin{lemma}
 \label{gen:lemma1}
Let $a,b\in\bR$ with $0\le a<b$.  There exists an $n$-system
$\uP\colon[a,b]\to\bR^n$ such that $\uP(a)=(a/n,\dots,a/n)$ and
$\uP(b)=(b/n,\dots,b/n)$.
\end{lemma}

\begin{proof}
Set $q_i=(n-i)a/n+ib/n$ for $i=0,\dots,n$ and, for each $j=1,\dots,n$,
define $P_j$ to be the unique continuous piecewise linear
function on $[q_0,q_n]=[a,b]$ which is constant equal to
$a/n$ on $[q_0,q_{n-j}]$, has slope $1$ on $[q_{n-j},q_{n-j+1}]$,
and is constant equal to $b/n$ on $[q_{n-j+1},q_n]$.  Then the map
$\uP=(P_1,\dots,P_n)\colon[a,b]\to\bR^n$ is an $n$-system
with the required properties. Its combined graph is illustrated on
Figure \ref{gen:lemma1:fig} below.
\end{proof}

\begin{figure}[h]
     \begin{tikzpicture}[xscale=0.4,yscale=0.7]
       \node[draw,circle,inner sep=1pt,fill] at (3,2) {};
         \draw[dashed] (3,2)--(3,1) node[below]{$q_{0}=a$};
       \node[draw,circle,inner sep=1pt,fill] at (8,2) {};
       \node[draw,circle,inner sep=1pt,fill] at (8,7) {};
         \draw[dashed] (8,7)--(8,1) node[below]{$q_{1}$};
       \node[draw,circle,inner sep=1pt,fill] at (13,2) {};
       \node[draw,circle,inner sep=1pt,fill] at (13,7) {};
         \draw[dashed] (13,7)--(13,1) node[below]{$q_{2}$};
       \node[draw,circle,inner sep=1pt,fill] at (23,2) {};
       \node[draw,circle,inner sep=1pt,fill] at (23,7) {};
         \draw[dashed] (23,7)--(23,1) node[below]{$q_{n-1}$};
       \node[draw,circle,inner sep=1pt,fill] at (28,7) {};
         \draw[dashed] (28,7)--(28,1) node[below]{$q_{n}=b$};
       \draw[thick] (3,2) -- (14,2);
        \draw[dotted, thick] (14,2) -- (15,2);
        \draw[dotted, thick] (21,2) -- (22,2);
        \draw[thick] (22,2) -- (23,2);
        \draw[dashed] (3,2)--(1,2) node[left]{$\disp \frac{a}{n}$};
       \draw[thick] (8,7) -- (14,7);
        \draw[dotted, thick] (14,7) -- (15,7);
        \draw[dotted, thick] (21,7) -- (22,7);
        \draw[thick] (22,7) -- (28,7);
        \draw[dashed] (8,7)--(1,7) node[left]{$\disp \frac{b}{n}$};
       \draw[thick] (3,2) -- (8,7);
       \draw[thick] (8,2) -- (13,7);
       \draw[thick] (13,2) -- (14,3);
        \draw[dotted, thick] (14,3) -- (15,4);
       \draw[thick] (22,6) -- (23,7);
        \draw[dotted, thick] (21,5) -- (22,6);
       \draw[thick] (23,2) -- (28,7);
       \node[below right] at (25.3,4.5) {$P_1$};
       \node[below right] at (20.3,4.5) {$P_2$};
       \node[below right] at (15.3,4.5) {$P_{n-2}$};
       \node[below right] at (10.3,4.5) {$P_{n-1}$};
       \node[below right] at (5.3,4.5) {$P_n$};
     \end{tikzpicture}
\caption{Combined graph of an $n$-system for Lemma \ref{gen:lemma1}.}
\label{gen:lemma1:fig}
\end{figure}
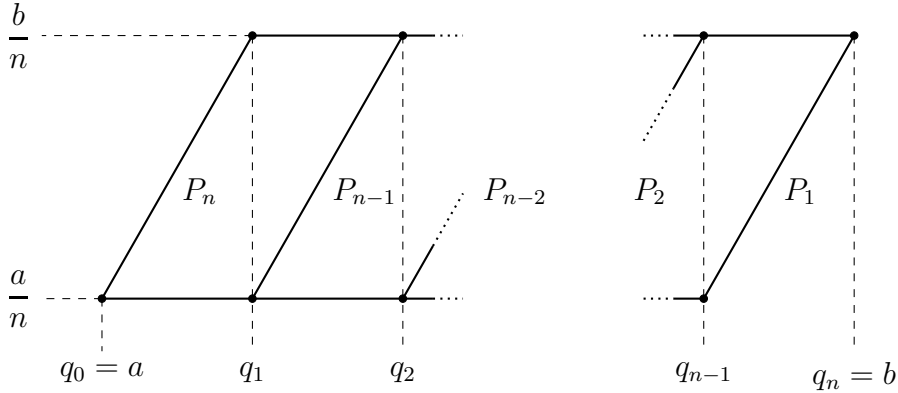

\begin{lemma}
 \label{gen:lemma2}
Let $a,b,c\in\bR$ with $0\le a<b<c$.  Suppose that
$\uP^{(1)}\colon[a,b]\to\bR^n$ and $\uP^{(2)}\colon[b,c]\to\bR^n$
are $n$-systems with $\uP^{(1)}(b)=\uP^{(2)}(b)=(b/n,\dots,b/n)$.
Then there is an $n$-system $\uP$ on $[a,c]$ which
restricts to $\uP^{(1)}$ on $[a,b]$ and to $\uP^{(2)}$ on $[b,c]$.
\end{lemma}

\begin{proof}
There is a unique map $\uP\colon [a,c]\to\bR^n$ which restricts to
$\uP^{(1)}$ on $[a,b]$ and to $\uP^{(2)}$ on $[b,c]$.  It is continuous
but not differentiable at $b$.  However, it satisfies the condition
(S3) at $q=b$ since all its coordinates are equal at that point.
Thus, it is an $n$-system.
\end{proof}

Let $I$ be any subinterval of $[0,\infty)$ and let $D$ be an arbitrary
discrete subset of $I$.  By combining Lemmas \ref{gen:lemma1} and
\ref{gen:lemma2}, we can construct an $n$-system $\uP$ on $I$
which takes the value $\uP(q)=(q/n,\dots,q/n)$ at each point $q$ of $D$.
Then, by choosing $D$ appropriately, we may ensure that $\sup_{q\in I}
\|\uP(q)-(q/n,\dots,q/n)\|_\infty$ is finite and arbitrarily small.
As we will see, a slightly more general argument shows that these
properties of approximation by $n$-systems extend to the following
maps, besides the map $\tP\colon I\to\bR^n$ sending each $q\in I$ to
$(q/n,\dots,q/n)$.

\begin{definition}
 \label{gen:def2}
Let $I$ be a subinterval of $[0,\infty)$ with non-empty
interior.  A generalized $n$-system on $I$ is a
continuous piecewise linear map
$\uP=(P_1,\dots,P_n)\colon I\to\bR^n$ with the following
properties.
\begin{itemize}
 \item[(G1)] For each $q\in I$, we have
   $0\le P_1(q)\le\cdots\le P_n(q)$ and $P_1(q)+\cdots+P_n(q)=q$.
 \item[(G2)] If $H$ is a non-empty open subinterval of $I$
   on which $\uP$ is differentiable, then there are integers
   $\rbot,\rtop$ with $1\le\rbot\le\rtop\le n$ such that
   $P_{\rbot},P_{\rbot+1},\dots,P_{\rtop}$ coincide on the whole
   interval $H$ and have slope $1/(\rtop-\rbot+1)$, while any other
   component $P_j$ of $\uP$ is constant on $H$.
 \item[(G3)] If $q$ is an interior point of $I$ at which $\uP$
   is not differentiable, if $\rbot,\rtop,\sbot,\stop$ are
   the integers for which
   \begin{equation}
    \label{gen:def2:eq}
     P'_j(q^-)=\frac{1}{\rtop-\rbot+1} \quad (\rbot\le j\le\rtop)
     \et
     P'_j(q^+)=\frac{1}{\stop-\sbot+1} \quad (\sbot\le j\le\stop),
   \end{equation}
   and if $\rbot<\stop$, then we have
   $P_{\rbot}(q)=P_{\rbot+1}(q)=\cdots=P_{\stop}(q)$.
\end{itemize}
\end{definition}

Figure \ref{gen:fig:G3}
shows the combined graph of the functions
$P_{\rbot},\dots,P_{\stop}$ over a neighborhood
of $q$ under the hypotheses of (G3) when $\rtop<\sbot$ and
$\rtop<\stop$.  The pattern is slightly different when
$\rbot\ge \sbot$ (resp.\ when $\rtop\ge \stop$): in that case,
there is no horizontal line segment to the right of $q$
(resp.\ to the left of $q$).  However, one cannot have
$\sbot\le \rbot<\stop\le \rtop$ and so at least one of the
inequalities $\rbot<\sbot$ or $\rtop<\stop$ must hold.

\begin{figure}[h]
     \begin{tikzpicture}[xscale=0.6,yscale=0.3]
       \node[draw,circle,inner sep=1pt,fill] at (6,4) {};
         \draw[dashed] (6,4)--(6,1) node[below]{$q$};
       \draw[thick] (3,0.7) -- (6,4) -- (9,6);
       \draw[thick] (3,4) -- (9,4);
         \node[left] at (3,0.7) {$P_{\rbot}=\cdots=P_{\rtop}$};
         \node[left] at (3,4) {$P_{\rtop+1}=\cdots=P_{\stop}$};
         \node[right] at (9,4) {$P_{\rbot}=\cdots=P_{\sbot-1}$};
         \node[right] at (9,6) {$P_{\sbot}=\cdots=P_{\stop}$};
     \end{tikzpicture}
\caption{Illustration for Condition (G3) when $\rtop<\sbot$
and $\rtop<\stop$.}
\label{gen:fig:G3}
\end{figure}

Note also that if \eqref{gen:def2:eq} holds at some point $q$
where $\uP$ is not differentiable, and if $\rbot\ge \stop$, then
we must have $\rbot>\stop$ and $P_{\rbot}(q)=\cdots=P_{\rtop}(q)
> P_{\sbot}(q)=\cdots=P_{\stop}(q)$.  We leave the verification
to the reader.

Clearly $n$-systems are also generalized $n$-systems.
The next result and its corollary formalize the claims
made just above Definition \ref{gen:def2}.

\begin{proposition}
 \label{gen:prop:approx}
Let $\tuP$ be a generalized $n$-system on a subinterval $I$
of $[0,\infty)$ with non-empty interior, and let $D$ be a
discrete subset of $I$.  Then there is an $n$-system $\uP$
on $I$ such that $\uP(t)=\tuP(t)$ for each $t\in D$.
\end{proposition}

\begin{proof}
Let $D_0$ be the set of points of $I$ where $\tuP$ is not
differentiable, including the boundary points of $I$ that
lie in $I$ (if any). Since $D_0$ is a discrete subset of
$I$, we may assume without loss of generality that $D$
contains $D_0$.  We may also assume that $I$ and $D$ have
the same infimum and the same supremum.

Let $(t_1,t_2)$ be a maximal subinterval of $I\setminus D$.
Then, $t_1$ and $t_2$ belong to $D$ and, as $D\subseteq I$,
it follows that $[t_1,t_2]$ is a subinterval of $I$.
As $\tuP$ is differentiable on $(t_1,t_2)$, there are
integers $\rbot$ and $\rtop$ with $1\le \rbot\le \rtop\le n$
such that $\tP_\rbot,\dots,\tP_\rtop$ coincide and have
slope $1/(\rtop-\rbot+1)$ on $(t_1,t_2)$ while all other
components $\tP_j$ are equal to a constant $c_j$ on that
interval.  Put
\[
 m=\rtop-\rbot+1
 \et
 c=(c_1+\cdots+c_{\rbot-1})+(c_{\rtop+1}+\cdots+c_n).
\]
By Lemma \ref{gen:lemma1}, there exists an $m$-system
$(A_1,\dots,A_m)$ on $[t_1,t_2]$ with
\[
 (A_1(t_i),\dots,A_m(t_i))
 = \Big( \frac{t_i}{m},\dots,\frac{t_i}{m} \Big)
 \quad \text{for $i=1,2$.}
\]
Since $\tP_j(q)=(q-c)/m$ for each $j=\rbot,\dots,\rtop$
and each $q\in[t_1,t_2]$, the map
$\uP\colon[t_1,t_2]\to\bR^n$ given by
\begin{equation}
 \label{gen:prop:approx:eq}
 \uP(q)
  = \Big( c_1, \dots, c_{\rbot-1},
          A_1(q)-\frac{c}{m},\dots,A_m(q)-\frac{c}{m},
          c_{\rtop+1}, \dots, c_n \Big)
 \quad (t_1\le q\le t_2),
\end{equation}
coincides with $\tuP$ at the points $q=t_1,t_2$.  As the
intervals $[t_1,t_2]$ cover $I$ and have no interior point
in common, this in fact defines a continuous piecewise
linear map $\uP\colon I\to\bR^n$ which coincides with
$\tuP$ on the set $D$.

To conclude, it remains simply to show that $\uP$ satisfies
the conditions (S1-S3) of Definition \ref{gen:def1}.
For each $q$ in a fixed interval $[t_1,t_2]$, the
formula \eqref{gen:prop:approx:eq} shows that the
coordinates of $\uP(q)$ sum up to $q$.  They also form
a monotone increasing sequence because
$A_1(q)\le\cdots\le A_m(q)$ are separately monotone
increasing functions of $q$ in $[t_1,t_2]$, and because
the coordinates of $\uP(t_i)=\tuP(t_i)$ form monotone
increasing sequences for $i=1,2$. Thus the condition (S1)
is fulfilled.  The condition (S2) also holds by construction.
The same is true for (S3) except possibly at the points $q$
of $D$ which lie in the interior of $I$.  The latter points
are the common boundary points $q=t_2$ of two maximal
subintervals $(t_1,t_2)$ and $(t_2,t_3)$ of $I\setminus D$.
Define $\rbot$, $\rtop$ as above and denote by $\sbot$,
$\stop$ the corresponding integers for the interval
$[t_2,t_3]$.  By construction, the function $\uP$ satisfies
$P'_\rbot(t_2^-)=P'_\stop(t_2^+)=1$. Suppose that $\rbot<\stop$.
If $t_2\in D_0$, we have $\tP_\rbot(t_2)=\cdots=\tP_\stop(t_2)$
because $\tuP$ satisfies the condition (G3) of
Definition \ref{gen:def2}.  If $t_2\notin D_0$, these
equalities still hold because then
$\rbot=\sbot$ and $\rtop=\stop$.  As $\uP(t_2)=\tuP(t_2)$,
we conclude that $P_\rbot(t_2)=\cdots=P_\stop(t_2)$
and thus (S3) is satisfied.
\end{proof}

\begin{corollary}
 \label{gen:cor}
Let $\tuP\colon I\to\bR^n$ be a generalized $n$-system
and let $\epsilon>0$.
Then there is an $n$-system $\uP$ on $I$ such that
$\|\tuP(q)-\uP(q)\|_\infty \le \epsilon$ for each $q\in I$.
\end{corollary}

\begin{proof}
Let $D$ be the set of all integer multiples of $\epsilon/2$
in $I$.  By Proposition \ref{gen:prop:approx}, there is an
$n$-system $\uP$ on $I$ such that $\uP(t)=\tuP(t)$ for
each $t\in D$.  Since the components of $\uP$ and of $\tuP$
have slopes between $0$ and $1$, we find that
\[
 \|\tuP(q)-\uP(q)\|_\infty
  \le \|\tuP(q)-\tuP(t)\|_\infty + \|\uP(q)-\uP(t)\|_\infty
  \le 2|q-t|
\]
for each $q\in I$ and each $t\in D$.  Upon choosing $t$ so that
$|q-t|\le \epsilon/2$, we conclude that $\uP$ has the required
property.
\end{proof}

Corollary \ref{gen:cor} shows that
the set $\cS_I$ of $n$-systems on a subinterval $I$ of $[0,\infty)$
is dense in the set $\cG_I$ of generalized $n$-systems on $I$, 
for the topology of uniform convergence.
More precisely, it can be shown that $\cG_I$ is the
completion of $\cS_I$ for that topology.
This justifies working with functions from this set.
In particular, the validity of Theorem \ref{gen:thm:nsys} extends to
generalized $n$-systems.

%
%

\section{A family of generalized $n$-systems}
 \label{sec:fam}

Again, we fix an integer $n\ge 2$.  Consider the set
\[
 \Delta^{(n)}
 = \{ (a_1,\dots,a_n)\in\bR^n \,;\, 0<a_1<\cdots<a_n
      \text{ and } a_1+\cdots+a_n=1 \}
\]
and its topological closure
\[
 \bar{\Delta}^{(n)}
 = \{ (a_1,\dots,a_n)\in\bR^n \,;\, 0\le a_1\le\cdots\le a_n
      \text{ and } a_1+\cdots+a_n=1 \}.
\]
By condition (G1) from Definition \ref{gen:def2}, we have
$q^{-1}\uP(q)\in \bar{\Delta}^{(n)}$ for any
generalized $n$-system $\uP$ and any $q>0$ in its
interval of definition.
For each $j=1,\dots,n$, we define a map
$\psi_j\colon \bar{\Delta}^{(n)} \to \bR$ by
\[
 \psi_j(a_1,\dots,a_n)=a_1+\cdots+a_j.
\]
The proof of our main result relies on the following basic
construction.

\begin{proposition}
 \label{fam:prop:basic}
Let $\ua=(a_1,\dots,a_n)\in\Delta^{(n)}$.  Define
\begin{align*}
 q_i&=a_1+\cdots+a_i+(n-i)a_i \quad (1\le i\le n),\\
 q_{n-1+i}&=(i-1)a_i+a_i+\cdots+a_n \quad (1\le i\le n).
\end{align*}
Then there exists a unique generalized $n$-system
$\uP=(P_1,\dots,P_n)$ on
$[q_1,q_{2n-1}]=[na_1,na_n]$ whose combined graph is as shown
on Figure \ref{gen:ex1:fig} below. For each $j=1,\dots,n-1$,
it has
\[
 \inf \psi_j\big(q^{-1}\uP(q)\big)
 = \psi_j(\ua)
 \et
 \sup \psi_j\big(q^{-1}\uP(q)\big)
 = \frac{j}{n}
\]
where both infimum and supremum are taken over all
$q\in[na_1,na_n]$.
\end{proposition}

In Figure \ref{gen:ex1:fig},
the number $1/m$ next to each slanted line segment indicates
its slope and thus the number $m$ of components of $\uP$ whose graph
coincide with this line segment over the corresponding interval
$[q_i,q_{i+1}]$.

\begin{figure}[h]
     \begin{tikzpicture}[scale=0.35]
       \node[draw,circle,inner sep=1pt,fill] at (4,3) {};
         \draw[dashed] (4,3)--(4,2) node[below]{$q_{1}$};
       \node[draw,circle,inner sep=1pt,fill] at (9,5) {};
         \draw[dashed] (9,5)--(9,2) node[below]{$q_{2}$};
       \node[draw,circle,inner sep=1pt,fill] at (13,7) {};
         \draw[dashed] (13,7)--(13,2) node[below]{$q_{3}$};
       \node[draw,circle,inner sep=1pt,fill] at (18,11) {};
         \draw[dashed] (18,11)--(18,2);
         \node[below] at (17.5,2) {$q_{n-1}$};
       \node[draw,circle,inner sep=1pt,fill] at (20,15) {};
       \node[draw,circle,inner sep=1pt,fill] at (20,3) {};
         \draw[dashed] (20,15)--(20,2);
         \node[below] at (19.75,2) {$q_{n}$};
       \node[draw,circle,inner sep=1pt,fill] at (21,5) {};
         \draw[dashed] (21,5)--(21,2);
         \node[below] at (21.6,2) {$q_{n+1}$};
       \node[draw,circle,inner sep=1pt,fill] at (23,7) {};
         \draw[dashed] (23,7)--(23,2);
         \node[below] at (23.9,2) {$q_{n+2}$};
       \node[draw,circle,inner sep=1pt,fill] at (30,11) {};
         \draw[dashed] (30,11)--(30,2) node[below]{$q_{2n-2}$};
       \node[draw,circle,inner sep=1pt,fill] at (40,15) {};
         \draw[dashed] (40,15)--(40,2) node[below]{$q_{2n-1}$};
       \draw[dashed] (4,3)--(2,3) node[left]{$a_{1}$};
       \draw[dashed] (9,5)--(2,5) node[left]{$a_{2}$};
       \draw[dashed] (13,7)--(2,7) node[left]{$a_{3}$};
       \node[left] at (1.5,9) {$\vdots$};
       \draw[dashed] (18,11)--(2,11) node[left]{$a_{n-1}$};
       \draw[dashed] (20,15)--(2,15) node[left]{$a_{n}$};
       \draw[thick] (4,3) -- (20,3);
       \draw[thick] (9,5) -- (21,5);
       \draw[thick] (13,7) -- (23,7);
       \draw[thick] (18,11) -- (30,11);
       \draw[thick] (20,15) -- (40,15);
       \draw[thick] (20,3) -- (21,5) -- (23,7) -- (24,23/3);
       \draw[dotted, thick] (24,23/3) -- (25,25/3);
       \draw[dotted, thick] (27,19/2) -- (28,10);
       \draw[thick] (28,10) -- (30,11) -- (40,15);
       \draw[thick] (4,3) -- (9,5) -- (13,7) -- (14,23/3);
        \node[font=\footnotesize] at (4.5,4.2) {$\frac{1}{n-1}$};
        \node[font=\footnotesize] at (9.3,6.2) {$\frac{1}{n-2}$};
        \node[font=\footnotesize] at (13.2,8.3) {$\frac{1}{n-3}$};
        \node[font=\footnotesize] at (15.9,10) {$\frac{1}{2}$};
        \node[font=\footnotesize] at (18.2,13) {$1$};
       \draw[dotted, thick] (14,23/3) -- (15,25/3);
       \draw[dotted, thick] (16.3,9.3) -- (17,10);
       \draw[thick] (17,10) -- (18,11) -- (20,15) -- (40,15);
        \node[font=\footnotesize] at (20.6,3.2) {$1$};
        \node[font=\footnotesize] at (22.3,5.3) {$\frac{1}{2}$};
        \node[font=\footnotesize] at (24.6,7.2) {$\frac{1}{3}$};
        \node[font=\footnotesize] at (28,9) {$\frac{1}{n-2}$};
        \node[font=\footnotesize] at (36,12.5) {$\frac{1}{n-1}$};
       \end{tikzpicture}
\caption{The combined graph of a generalized $n$-system.}
\label{gen:ex1:fig}
\end{figure}
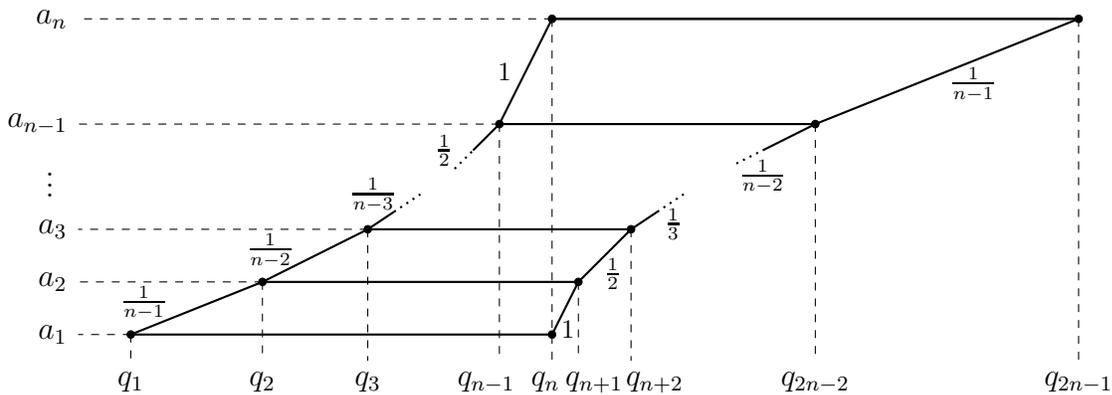

The existence and uniqueness of $\uP$ is easily seen.  In order
to compute the required infimum and supremum, we use the following
result.

\begin{lemma}
 \label{fam:lemmaM}
Let $0<a<b$, let $M\colon [a,b]\to\bR$ be a continuous piecewise
linear function on $[a,b]$, let $q_1=a<q_2<\cdots<q_t=b$ be the
points where $M$ is not differentiable (including $a$ and $b$),
and let $\rho_i$ denote the constant value of the derivative
of $M$ on $(q_i,q_{i+1})$ for $i=1,\dots,t-1$.
\begin{itemize}
 \item[1)] If $M(a)/a>\rho_1>\cdots>\rho_{t-1}$, then $M(q)/q$
  is strictly decreasing on $[a,b]$, and bounded below
  by $\rho_{t-1}$.
 \item[2)] If $\rho_1>\cdots>\rho_{t-1}>M(b)/b$, then $M(q)/q$
   is strictly increasing on $[a,b]$, and bounded above
   by $\rho_1$.
\end{itemize}
\end{lemma}

\begin{proof}
It suffices to prove this for $t=2$ because the general case
then follows by induction on $t$.  Assuming that $t=2$,
we can write
\[
 M(q)=M(a)+\rho_1(q-a)=M(b)+\rho_1(q-b) \quad (a\le q\le b).
\]
Under the hypothesis 1), we also have $M(a)>\rho_1a$ and
so $M(q)/q=\rho_1+(M(a)-\rho_1a)/q$ is a strictly decreasing
function of $q$ on $[a,b]$, bounded below by $\rho_1$.
Similarly, under the hypothesis 2), we have $M(b)<\rho_1b$,
then $M(q)/q=\rho_1+(M(b)-\rho_1b)/q$ is strictly increasing
on $[a,b]$ and bounded above by $\rho_1$.
\end{proof}

\begin{proof}[Proof of Proposition \ref{fam:prop:basic}]
As noted above, it suffices to prove the second assertion.
To this end, fix an index $j$ with $1\le j\le n-1$ and define
$M_j=P_1+\dots+P_j$.  Then $M_j$ is a continuous piecewise linear
map on $[q_1,q_{2n-1}]$ with slope $(j-i)/(n-i)$ on
$[q_i,q_{i+1}]$ for $i=1,\dots,j-1$, slope $0$ on $[q_j,q_n]$,
slope $1$ on $[q_n,q_{n+j}]$ and slope $j/i$ on
$[q_{n+i-1},q_{n+i}]$ for $i=j+1,\dots,n-1$.  Since
\[
 \frac{M_j(q_1)}{q_1}
  =\frac{j}{n}>\frac{j-1}{n-1}>\cdots>\frac{1}{n-j+1}>0
\]
and
\[
 1>\frac{j}{j+1}>\cdots>\frac{j}{n-1}>\frac{j}{n}
   =\frac{M_j(q_{2n-1})}{q_{2n-1}},
\]
we conclude from Lemma \ref{fam:lemmaM} that the ratio $M_j(q)/q$
is strictly decreasing on $[q_1,q_n]$ and strictly increasing on
$[q_n,q_{2n-1}]$.  This means that
\[
 \inf\frac{M_j(q)}{q}
  =\frac{M_j(q_n)}{q_n}
  =\frac{a_1+\cdots+a_j}{a_1+\cdots+a_n}
  =\psi_j(\ua)
 \et
 \sup \frac{M_j(q)}{q}
  =\frac{j}{n}.
 \qedhere
\]
\end{proof}

The last result of this section combines the above proposition
with the following observation.

\begin{lemma}
 \label{fam:lemma:rescaling}
If $\uP\colon [a,b]\to\bR^n$ is a generalized $n$-system with
$b>a>0$, then, for any $d>c>0$ with $d/c=b/a$, the map $\tuP
\colon [c,d]\to\bR^n$ given by
\[
 \tuP(q) = \frac{c}{a} \uP\Big(\frac{a}{c}q\Big)
 \quad \text{for each $q\in[c,d]$}
\]
is also a generalized $n$-system and we have both
\[
 \inf_{q\in[a,b]}\psi_j\big(q^{-1}\uP(q)\big)
  = \inf_{q\in[c,d]}\psi_j\big(q^{-1}\tuP(q)\big)
 \et
 \sup_{q\in[a,b]}\psi_j\big(q^{-1}\uP(q)\big)
  = \sup_{q\in[c,d]}\psi_j\big(q^{-1}\tuP(q)\big).
\]
\end{lemma}

The combined graph of this new map $\tuP$ is the image of
the combined graph of $\uP$ by uniform scaling of
ratio $c/a$.    Upon noting furthermore
that Lemma \ref{gen:lemma2} extends to generalized
$n$-systems, we conclude with the following result.

\begin{proposition}
 \label{fam:prop:pasting}
Let $E$ be an arbitrary non-empty subset of $\bar{\Delta}^{(n)}$.
There exists a generalized $n$-system $\uP$ on $[1,\infty)$
which, for each $j=1,\dots,n$, satisfies
\[
 \liminf_{q\to\infty} \psi_j\big(q^{-1}\uP(q)\big)
  = \inf \psi_j(E)
 \et
 \limsup_{q\to\infty} \psi_j\big(q^{-1}\uP(q)\big)
  = \frac{j}{n}.
\]
\end{proposition}

\begin{proof}
If $E$ consists of the single point $(1/n,\dots,1/n)$, we simply
take $\uP(q)=(q/n,\dots,q/n)$ for each $q\ge 1$.  Otherwise, we
choose a sequence $(\ua^{(i)})_{i\ge 1}$ of points of $\Delta^{(n)}$
whose set of accumulation points is the topological closure
$\bar{E}$ of $E$.  For each $i\ge 1$, we consider the generalized
$n$-system attached to $\ua^{(i)}$ by Proposition
\ref{fam:prop:basic}, and, starting with $q_1=1$, we use Lemma
\ref{fam:lemma:rescaling} recursively to transform it into a
generalized $n$-system $\uP^{(i)}$ on an interval of the form
$[q_i,q_{i+1}]$.  By construction, we have
\begin{equation}
  \label{fam:prop:pasting:eq}
 \inf_{q\in[q_i,q_{i+1}]} \psi_j\big(q^{-1}\uP^{(i)}(q)\big)
 = \psi_j(\ua^{(i)})
 \et
 \sup_{q\in[q_i,q_{i+1}]} \psi_j\big(q^{-1}\uP^{(i)}(q)\big)
 = \frac{j}{n}
\end{equation}
for $j=1,\dots,n$.  We also note that
\[
 \uP^{(i-1)}(q_i)=\uP^{(i)}(q_i)=(q_i/n,\dots,q_i/n)
 \quad (i\ge 2)
\]
and that $\limsup q_{i+1}/q_i >1$ since $E$ contains at least
one point $(a_1,\dots,a_n)$ with $a_n>a_1$.  Thus $\lim q_i=\infty$
and so there exists
a unique generalized $n$-system $\uP$ on $[1,\infty)$ which
restricts to $\uP^{(i)}$ on $[q_i,q_{i+1}]$ for each $i\ge 1$.
In view of \eqref{fam:prop:pasting:eq}, we obtain
\[
 \liminf_{q\to\infty} \psi_j\big(q^{-1}\uP(q)\big)
  = \liminf_{q\to\infty} \psi_j(\ua^{(i)})
  = \inf \psi_j(E)
  \et
  \limsup_{q\to\infty} \psi_j\big(q^{-1}\uP(q)\big)= \frac{j}{n}
\]
for each $j=1,\dots,n$, as requested.
\end{proof}

%
%

\section{Proof of the main result}
 \label{sec:proof}

The next result is the last piece that we need in order to
prove Theorem \ref{results:thm:main}.

\begin{proposition}
 \label{proof:prop1}
Let $n\in\bN^*$. Suppose that $\psibot_1,\dots,\psibot_n
\in [0,1]$ satisfy $\psibot_n\le n/(n+1)$,
\begin{equation}
 \label{proof:prop1:eq}
 \frac{\psibot_j}{j} \le \frac{\psibot_{j+1}}{j+1}
 \et
 \frac{1-\psibot_j}{n+1-j} \le \frac{1-\psibot_{j+1}}{n-j}
 \quad
 (1\le j\le n-1).
\end{equation}
Then, there exists a finite non-empty subset $E$ of
$\bar{\Delta}^{(n+1)}$ such that
\[
 \psibot_j=\min \psi_j(E) \quad (1\le j\le n).
\]
\end{proposition}

\begin{proof}
If $n=1$, we have $\psibot_1\le 1/2$ and we simply take
$E=\{(\psibot_1,1-\psibot_1)\}$.  Suppose from now on that
$n\ge 2$ and let $k\in\{1,\dots,n-1\}$.  We claim that
there exists $\ua\in\bar{\Delta}^{(n+1)}$ such that
$\psi_j(\ua)\ge \psibot_j$ for $j=1,\dots,n$, with equality
for $j=k$ and $j=k+1$.  To show this, set $c=\psibot_k/k$ and
$d=(1-\psibot_{k+1})/(n-k)$, and consider the point
\[
 \ua
  = \big(
     \overbrace{c,\dots,c}^{\text{$k$ times}},\
     \psibot_{k+1}-\psibot_k,
     \overbrace{d,\dots,d}^{\text{$n-k$ times}}
    \big)
  \in \bR^{n+1}.
\]
By hypothesis, we have $c\ge 0$ and the inequalities
\eqref{proof:prop1:eq} for $j=k$ yield
$c \le \psibot_{k+1}-\psibot_k \le d$.  As the sum of
the coordinates of $\ua$ is $1$, this means that
$\ua\in\bar{\Delta}^{(n+1)}$.  We also note from
\eqref{proof:prop1:eq} that the ratios $\psibot_j/j$
and $(1-\psibot_j)/(n+1-j)$ are both monotone increasing
functions of $j$ for $j\in\{1,\dots,n\}$.  Thus,
for $j=1,\dots,k$, we have $\psibot_j/j\le c$
and so we find $\psi_j(\ua)=jc \ge \psibot_j$, with
equality if $j=k$.  Similarly, for $j=k+1,\dots,n$,
we have $d\le (1-\psibot_j)/(n+1-j)$ and so we
obtain $\psi_j(\ua)=1-(n+1-j)d \ge \psibot_j$,
with equality if $j=k+1$.  This proves our claim.
By varying $k$, this produces a finite set of points
with the required property.
\end{proof}

In general, we cannot expect that the set $E$ consists
of a single element because each $\ua=(a_1,\dots,a_{n+1})$
in $\bar{\Delta}^{(n+1)}$ satisfies
\[
 \psi_{j-1}(\ua)+\psi_{j+1}(\ua)-2\psi_j(\ua)
  = a_{j+1}-a_j
  \ge 0
 \quad
 (2\le j\le n),
\]
and so, for example, there is no point $\ua$ in
$\bar{\Delta}^{(4)}$ with $\psi_1(\ua)=0$, $\psi_2(\ua)=1/3$
and $\psi_3(\ua)=1/2$, although the numbers $\psibot_1=0$,
$\psibot_2=1/3$ and $\psibot_3=1/2$ satisfy the conditions
\eqref{proof:prop1:eq} for $n=3$.

\begin{proof}[Proof of Theorem \ref{results:thm:main}]
Suppose that $\omega_0,\dots,\omega_{n-1}\in[0,\infty]$
satisfy the conditions of the theorem.  Then the numbers
\[
 \psibot_j = \frac{1}{\omega_{n-j}+1} \in [0,1]
 \quad (1\le j\le n),
\]
fulfill the hypotheses of the
above proposition.  Let $E$ be a corresponding subset of
$\bar{\Delta}^{(n+1)}$, as in that proposition, and let $\uP$
be a generalized $(n+1)$-system on $[1,\infty)$ as constructed
by Proposition \ref{fam:prop:pasting} for this choice of $E$.
For $j=1,\dots,n$, we have
\[
 \liminf_{q\to\infty} \psi_j\big(q^{-1}\uP(q)\big)
  = \inf \psi_j(E) = \psibot_j
 \et
 \limsup_{q\to\infty} \psi_j\big(q^{-1}\uP(q)\big)
  = \frac{j}{n+1}.
\]
Corollary \ref{gen:cor} and Theorem \ref{gen:thm:nsys} provide
a point $\uu\in\bR^{n+1}\setminus\{0\}$ for which
the difference $\uP-\uL_\uu$ is bounded on $[1,\infty)$.
By the above, and by definition of the quantities
$\psibot_j(\uu)$ and $\psitop_j(\uu)$ in
Section \ref{sec:link}, this point satisfies
\[
 \psibot_j(\uu)=\psibot_j
 \et
 \psitop_j(\uu)=\frac{j}{n+1}
 \quad
 (1\le j\le n).
\]
By Proposition \ref{link:prop:omega}, this means that,
as requested, we have, for $j=0,\dots,n-1$,
\[
 \omega_j(\uu)
  =\frac{1}{\psibot_{n-j}(\uu)}-1
  =\omega_j
 \et
 \homega_j(\uu)
  =\frac{1}{\psitop_{n-j}(\uu)}-1
  =\frac{j+1}{n-j}.
\]
Since $\homega_{n-1}(\uu)=n<\infty$, the coordinates of
$\uu$ must be linearly independent over $\bQ$.
\end{proof}

\medskip
As a complement, note that each point $\ua=(a_1,\dots,a_{n+1})$
in $\bar{\Delta}^{(n+1)}$ satisfies
\begin{equation}
 \label{proof:complement:eq}
 0 \le \frac{\psi_j(\ua)}{j} \le \frac{\psi_{j+1}(\ua)}{j+1}
 \et
 0 \le \frac{1-\psi_j(\ua)}{n+1-j} \le \frac{1-\psi_{j+1}(\ua)}{n-j}
 \quad (1\le j\le n-1).
\end{equation}
Based on this, Theorem \ref{results:thm:SL} can be given the
following short proof.  Let $\uu\in\bR^{n+1}\setminus\{0\}$.
By Theorem \ref{gen:thm:nsys}, there exists
$q_0>0$ and an $(n+1)$-system $\uP$ on $[q_0,\infty)$ such that
$\uP-\uL_\uu$ is bounded on $[q_0,\infty)$.  For such a choice
of $\uP$, we have
\[
 \psibot_j(\uu)
  = \liminf_{q\to\infty} \psi_j\big(q^{-1}\uP(q)\big)
 \quad
 (1\le j\le n).
\]
Since $q^{-1}\uP(q)$ belongs to $\bar{\Delta}^{(n+1)}$ for
each $q\ge q_0$, the inequalities \eqref{proof:complement:eq}
apply to these points.  They imply that
\[
 0 \le \frac{\psibot_j(\uu)}{j}
   \le \frac{\psibot_{j+1}(\uu)}{j+1}
 \et
 0 \le \frac{1-\psibot_j(\uu)}{n+1-j}
   \le \frac{1-\psibot_{j+1}(\uu)}{n-j}
 \quad (1\le j\le n-1),
\]
and the conclusion follows using Proposition
\ref{link:prop:omega} to translate these estimates in terms
of $\omega_0(\uu),\dots,\omega_{n-1}(\uu)$.



\begin{thebibliography}{99}
%
\bibitem{BL2010}
  Y.~Bugeaud and M.~Laurent,
  On transfer inequalities in Diophantine
  approximation II,
  {\it Math.\ Z.\ }{\bf 265} (2010), 249--262.
%
\bibitem{Di1842}
  L.~G.~P.~Dirichlet,
  Verallgemeinerung eines Satzes aus der Lehre von Kettenbr\"uchen
  nebst einigen Anwendungen auf die Theorie der Zahlen,
  {\it SBer.\ Kgl.\ Preuss.\ Akad.\ Wiss.\ Berlin}, 1842, 93--95;
  reproduced in {\it G.~Lejeune Dirichlet's Werke}, vol.~1,
  Cambridge U.~Press (2012), 633--638.
%
\bibitem{GL1987}
  P.~M.~Gruber and C.~G.~Lekkerkerker,
  {\it Geometry of numbers},
  North-Holland, 1987.
%
\bibitem{Ja1938}
  V.~Jarn\'{\i}k,
  Zum Khintchineschen \"Ubertragungssatz,
  {\it Trav.\ Inst.\ Math.\ Tbilissi} {\bf 3} (1938), 193--212.
%
\bibitem{Kh1926a}
  A.~Y.~Khintchine, Zur metrischen Theorie der diophantischen Approximationen,
  {\it Math.\ Z.\ }{\bf 24} (1926), 706--714.
%
\bibitem{Kh1926b}
  A.~Y.~Khintchine, \"Uber eine Klasse linearer diophantischer Approximationen,
  {\it Rendiconti Circ.\ Mat.\ Palermo} {\bf 50} (1926), 170--195.
%
\bibitem{La2009}
  M.~Laurent, Exponents of Diophantine approximation in dimension two,
  {\it Can.\ J.\ Math.\ }{\bf 61} (2009), 165--189.
%
\bibitem{La2009b}
 M.~Laurent, On transfer inequalities in Diophantine approximation,
  in: {\it Analytic Number Theory, Essays in honour of Klaus Roth},
  Cambridge U.\ Press (2009), 306--314.
%
\bibitem{Ma1937}
  K.~Mahler, Neuer Beweis eines Satz von A.~Khintchine,
  {\it Mat.\ Sbornik\ }{\bf 43} (1936), 961--962.
%
\bibitem{R2015}
  D.~Roy,
  On Schmidt and Summerer parametric geometry of numbers,
  {\it Ann.\ of Math.\ }{\bf 182} (2015), 739--786.
%
\bibitem{Sc1967}
  W.~M.~Schmidt,
  On heights of algebraic subspaces and Diophantine approximations,
  {\it Ann.\ of Math.\ } {\bf 85} (1967), 430--472.
%
\bibitem{SS2009}
  W.~M.~Schmidt and L.~Summerer,
  Parametric geometry of numbers and applications, \textit{Acta Arith.\ }\textbf{140} (2009), 67--91.
%
\bibitem{SS2013a}
  W.~M.~Schmidt and L.~Summerer,
  Diophantine approximation and parametric geometry of numbers, \textit{Monatsh.\ Math.\ }\textbf{169} (2013), 51--104.
%
\end{thebibliography}
\end{document}